\documentclass[reqno,10pt]{amsart}
\usepackage{color}
\usepackage{amssymb,amsmath,amsthm,amstext,amsfonts}
\usepackage[dvips]{graphicx}
\usepackage{psfrag}

\makeatletter \@addtoreset{equation}{section} \makeatother

\renewcommand\thetable{\thesection.\@arabic\c@table}

\theoremstyle{plain}
\newtheorem{maintheorem}{Theorem}
\newtheorem{theorem}{Theorem }[section]

\newtheorem{corollary}[theorem]{Corollary}
\newtheorem{maincorollary}{Corollary}

\theoremstyle{definition} \theoremstyle{remark}
\newtheorem{remark}[theorem]{Remark}

       \newcommand{\Si}{\Sigma}

\newcommand{\R}{\mathbb{R}}

\newcommand{\Leb}{Leb}

\newcommand{\cM}{\mathcal{M}}

\newcommand{\cA}{\mathcal{A}}

\begin{document}

\title[Kac's lemma for suspension flows]{A version of Kac's lemma on first return  \\ times for suspension flows}

\author{Paulo Varandas}

\address{Departamento de Matem\'atica, Universidade Federal da Bahia\\
  Av. Ademar de Barros s/n, 40170-110 Salvador, Brazil \& CMUP, University of Porto.}
\email{paulo.varandas@ufba.br}

\date{\today}

\begin{abstract} 
In this article we study the mean return times to a given set for suspension flows. In the discrete time
setting, this corresponds to the classical version of Kac's lemma \cite{K} that the mean of the first return
time to a set with respect to the normalized probability measure is one. In the case of suspension flows we 
provide formulas to compute the mean return time. Positive measure sets on cross sections are also considered.In particular, this varies linearly with continuous
reparametrizatons of the flow and takes into account the mean escaping time from the original set.
Relation with entropy and returns to positive measure sets on cross sections is also considered.
\end{abstract}



\maketitle

\section{Introduction}

The construction of invariant measures and the study of their
statistical properties are the main goals of the ergodic theory in the 
attempt to describe the behavior of a discrete and continuous time 
dynamical system. 
In the presence of invariant measures the celebrated Poincar\'e' s recurrence theorem
guarantees that recurrence occurs almost everywhere and, from this qualitative statement,
a natural question is to ask for quantitative results on the recurrence. Given a measure preserving map 
$f: (X,\mu) \to (X,\mu)$ and a positive measure set $A\subset X$ consider the 
\emph{first return time} to the set $A$ given by
$$
n_A(x)=\inf\{k\ge 1 : f^k(x)\in A\}.
$$
In 1947, Kac \cite{K} established an expression to the expectation of the return time proving that
$$
\int_A n_A(x) \,d\mu_A= 1,
$$
where $\mu_A(\cdot) =\mu(\cdot \cap A)/ \mu(A)$ denotes the normalized restriction of $\mu$ to $A$.
Such a quantitative estimate has been very useful tool in the study of other recurrence properties
(see e.g. \cite{CR,H,HP,R,S} and the references therein)

In the time-continuous setting the situation is substantially different. It is clear that if a flow $(X_t)_t$ preserves
a probability measure $\mu$ then it is an invariant probability measure for discrete-time 
transformation obtained as the time-$t$ map $f=X_t$ of the flow and Kac' s lemma holds for $(f,\mu)$. 
However, this does not constitute a true time-continuous quantitative recurrence estimate. In fact, given a
continuous flow $(X_t)_t$, an open set $A$ and a point $x\in A$, the first return time of $x$ to the set $A$ should 
be considered after the time $e_A(x,t)>0$ that orbit of the point $x$ needs to exit $A$ 
(c.f. Section~\ref{s.suspension.flows} for the precise definition).

Our purpose here is to prove a time-continuous quantitative Kac-like formula for first return times for
suspension flows for positive measure sets both in the cross-section and ambient spaces. 
Our first motivation follows from Ambrose and Kakutani \cite{Am,AK} which proved that 
any aperiorid flow $(Y_t)_t$ on a probability space $(M,\nu)$ without singularities is  metrically isomorphic to a suspension
flow $(Z_t)_t$. More precisely, there exists a measurable transformation $f: \Sigma \to \Sigma$, an 
$f$-invariant probability measure $\mu_\Sigma$ and a roof function $\tau:\Sigma \to \mathbb R_0^+$ defining 
a suspension flow $(Z_t)_t$ on $\Sigma_\tau$ that preserves $\mu=\mu_\Sigma \times \Leb / \int \tau d\mu_\Sigma$, and there exists a measure preserving $\psi: (M,\nu) \to (\Sigma_\tau,\mu)$ that is a bijection (modulo 
zero measure sets) and satisfying 
$
Z_t\circ \psi = \psi \circ Y_t
$
and
$
	Y_t\circ \psi^{-1} = \psi^{-1} \circ Z_t
$
for all $t\ge 0$.
Our second motivation is that, using finitely many tubular neighborhoods or the Hartman-Grobman
theorem around singularities, smooth flows on compact manifolds that exhibit some hyperbolicity 
often can be modelled by suspension flows. This is true for hyperbolic flows (see e.g. \cite{Bo73,Rt}) 
for geometric Lorenz and three-dimensional singular-hyperbolic flows  (see e.g. \cite{AP}) 
and for smooth three-dimensional flows with positive topological entropy (see \cite{LS}).  
A final motivation is that it is expected to have a large utility and large amount of applications, 
including the study of large deviations for return times in cylinders, fluctuations theorems
and study of the recurrence properties for Axiom A flows. 
Our main results below assert that the expectation of the first return time to open sets takes into account
the mean return time of the roof function. Moreover, the mean return time to subsets of the global cross section
can be understood as a quotient of the corresponding entropies of the flow and the first return time map on 
the base. The detailed statements will be given in the next section.

After this work was complete we were informed by J.-R. Chazottes  
of the unnoticed and very interesting work by G. Helberg~\cite{Hel} where the author addresses the problem 
of return time estimates for positive measure sets for measurable and continuous flows. Since this is a 
rather unknown result, never quoted, we shall describe his main results in Section~\ref{s.suspension.flows} for
completeness.

\section{Suspension flows}\label{s.suspension.flows}

Given a topological space $\Sigma$, a measurable map $f: \Si \to \Si$ and a
roof function $\tau:\Si \to \R^{+}$ that is bounded away from zero consider the quotient space
$$
\Sigma_\tau=\{(x,s) \in \Si \times \R^{+} : 0 \leq s \leq \tau(x)\} / \sim
$$
obtained by the equivalence relation that $(x,\tau(x)) \sim (f(x),0)$ for every $x\in \Sigma$. The
suspension flow $(X_t)_t$ on $\Sigma_\tau$ associated to $(f,\Sigma,\tau)$ is defined by 
as the 'vertical displacement' $X_t(x,s)=(x,t+s)$ whenever the expression is well defined. More precisely,
\begin{equation}\label{eq:1}
X_s(x,t)
	=\Big(f^k(x), t+s-\sum_{j=0}^{k-1} \tau(f^j(x)) \Big)
\end{equation}
where $k=k(x,t,s)\ge 0$ is determined by
$
\sum_{j=0}^{k-1} \tau(f^j(x)) \leq t+s < \sum_{j=0}^{k} \tau(f^j(x)).
$
We shall refer to $\Sigma$ as a cross-section to the flow.
Since $\tau$ is bounded away from zero there
is a natural identification between the space $\cM_X$ of
$(X_t)_t$-invariant probability measures and the space $\cM_f(\tau)$ of
$f$-invariant probability measures $\mu$ so that $\tau \in L^1(\mu)$.
Namely, if $m$ denotes the Lebesgue measure in $\R$ then the map
\begin{equation}\label{eq:normal}
\begin{array}{cccc}
L: & \cM_f(\tau) &    \to  &  \cM_X\\
   &     \mu       & \mapsto & \bar\mu := \frac{(\mu \times m)|_{\Sigma_\tau}}{\int_\Sigma \tau \, d\mu }
\end{array}
\end{equation}
is a bijection. 
For that reason in some situations one reduces some ergodic properties from the suspension semiflow to
the first return map to the global Poincar\'e section $\Sigma\times\{0\}$. For instance, while the measure 
induced on the cross section can detect ergodicity of the flow it fails to detect the mixing properties. 

From now on, and if otherwise stated, we shall fix an $f$-invariant ergodic probability measure $\mu$ 
and let $\bar\mu$ denote the corresponding flow-invariant ergodic probability measure.
We start by recalling Kac's qualitative recurrence estimates for measurable maps. 

\begin{theorem}[Kac's lemma]\label{leKac}
Let $f: \Sigma \to\Sigma$ be a measurable map preserving a probability measure $\mu$. For any 
measurable set $A$ with $\mu(A)>0$ the hitting time $n_A(\cdot): \Sigma \to \mathbb N$ 
defined by
$
n_A(x)=\inf\{ k\ge 1 : f^k(x) \in A\}
$
satisfies 
$$
\int_A n_A(x) \, d\mu=1
$$
Equivalently, $\int_A n_A(x) \, d\mu_A=\frac1{\mu(A)}$ is inversely proportional to the measure of $A$. 
\end{theorem}

In the end of the sixties, Helmberg~\cite{Hel} studied recurrence for arbitrary sets satisfying 
a 'boundary condition' with respect to measurable and continuous flows. In general, if a set $A$ 
is such that the boundary $\partial A$ has positive $\bar\mu$-measure and admits a fractal structure 
then it was not clear how to define properly return times to $A$.
To overcome this difficulty Helmberg paper defined return times using (measurable) families of
exit and entrance regions: for any $s>0$ define the \emph{exit region} $A_s$ by
$
A_s = \{ z \in M : \exists 0\le \ell < r \le s \text{ so that } X_\ell(z) \in A \text{ and } X_r(z) \notin A\}
$ 
and analogousy the \emph{entrance region} $\tilde A_s$ by
$
\tilde A_s = \{ z \in M : \exists 0\le \ell < r \le s \text{ so that } X_\ell(z) \notin A \text{ and } X_r(z) \in A\}.
$ 
These are cleary decreasing families as $s$ tends to zero. In this case, the \emph{return time} is 
defined by $r_A(z)= \inf\{ s>0 : z\in \tilde A_s\}$ and it is proved to be almost everywhere well defined. 
More precisely,
\begin{theorem}\cite[Theorem~5]{Hel}\label{Hel}
Let $(X_t)_{t\ge 0}$ be a semiflow of measurable transformations on a probability space 
$(M, \bar\mu)$. If $A\subset M$ is a positive $\bar\mu$-measure set 
and $t>0$ is such that $\lim_{n\to\infty} \bar\mu (A_{\frac{t}{n}})=0$ then
$$
\lim_{s \to 0} \frac{1}{s} \int_{A_s} r_A (z) \, d\bar\mu
	= \bar\mu \bigg(  \bigcup_{r\ge 0} X_{-r} (A) \bigg) - \bar\mu (A).
$$
In particular, if $\bar\mu$ is ergodic then the right hand side is equal to $1-\bar\mu(A)$.
\end{theorem}
The limit in the previous theorem follows by the strategy used by Helmberg of approximating return times 
for the flow by studying return times for individual time-$s$ maps $T_s$ for small $s>0$.  While the result holds under big generality e.g. for all positive measure sets $A$ so that $\bar\mu (\partial A)=0$ it provides 
estimate on the limiting  mean return time function on the exit components $A_s$, that is, on points in the eminence 
of leaving the set $A$.

In our time-continuous setting of suspension flows and geometrical objects like cylinders or balls whose
boundary will have always always zero. We will compute the mean return time among 
points in the whole set $A$. Taking into account the normalization given by equation~\eqref{eq:normal} it is a natural question to understand in which sense the mean return time for suspension flows reflects its dependence on the integral of the roof function. More precisely, given a set $A$ if one makes 
a reparametrization of the suspension flow leaving the set $A$ invariant it is interesting to understand if the
mean return times for both flows coincide or, if not, if they can be related to each other.  
Moreover, we shall address some questions on positive measures sets on the cross section $\Sigma$, which
have zero $\bar\mu$-measure on the flow while are dynamically significant and contain non-trivial recurrence. 

Let us introduce some necessary notions.
Given a set $A\subset \Sigma_\tau$ define the \emph{escape time} $e_A(\cdot, \cdot) : \Sigma_\tau \to \mathbb R_0^+$
to be given by 
$$
e_A(x,t)=\inf\{ s>0 : X_s(x,t)\not\in A\}
$$
for any $(x,t)\in \Sigma_\tau$. The escape time $e_A(x,t)$ is clearly zero for any point $(x,t)$ in the 
complement of $A$.  
We define also the \emph{hitting time} to $A$ as the function $n_A : \Sigma_\tau \to \R_0^+$
given by 
\begin{equation}\label{def:entrance}
n_A(x,t)=\inf\{ s>e_A(x,t) : X_s(x,t)\in A\}.
\end{equation}
The \emph{return time function to $A$} consists of the restriction of $n_A(\cdot,\cdot)$  to the set $A$ and,
clearly, $n_A(x,t) \ge e_A(x,t)$ for every $(x,t)\in A$.
Alternatively,  we can define the return time 
$
\tilde n_A(x,t)=\inf\{ s>0 : X_{s+ e_A(x,t)} (x,t)\in A\},
$
in which case the functions verify $n_A=e_A + \tilde n_A$ and consequently to $ \int_A n_A \,d\bar\mu$ and
$ \int_A \tilde n_A \,d\bar\mu$ just differ by the average escaping time $ \int_A e_A \,d\bar\mu$.
For that reason we shall restrict to the study of the previously defined hitting time function. 
Our first result is a version of Kac's lemma for returns to cylinder sets.

\begin{maintheorem}\label{thm:Kac}
Let $f: \Sigma\to \Sigma$ be a measurable invertible map on a topological space $\Sigma$ preserving 
an ergodic probability measure $\mu$ and let $\tau: \Sigma \to \mathbb R_0^+$ be 
a $\mu$-integrable roof function. Denote
by $(X_t)_t$ the suspension semiflow over $(f,\mu,\tau)$ and let $\bar \mu$ be the corresponding 
$(X_t)_t$ invariant probability measure given by \eqref{eq:normal}. Given 
a positive $\bar\mu$-measure cylinder set $A= I \times [t_1,t_2] \subset \Sigma\times \mathbb R^+$ 
 in $\Sigma_\tau$ then 
\begin{align*}
\int_A   n_A(x,t) \; d\bar\mu_A
	 & =   \int_A e_A(x,t)   \, d\bar \mu_A + \big( 1- \bar\mu(A) \big) \frac{1}{\mu(I)} \, \int_\Sigma \tau \,d\mu \\
	 & =  \frac{m([t_1,t_2])}{2} + \frac{1}{\mu(I)} \int_\Sigma \bigg( \tau  - (t_2-t_1) \chi_I \bigg) \, d\mu
\end{align*}
In particular, if the set $A$ is kept unchanged and we change the roof function $\tau$ the mean 
return time $\int_A n_A(x,t) \; d\mu_A$ varies lineary with respect to the  mean of the roof function
$\int_\Sigma  \tau \, d\mu$.
\end{maintheorem}

\begin{proof}
Let $n_I(\cdot)$ be the first return time function
to $I$ by $f$. Given $(x,t)\in\Sigma_\tau$ and $s\ge 0$ it follows from \eqref{eq:1}
that 
$
X_s(x,t)
	=(f^k(x), t+s-\sum_{j=0}^{k-1} \tau(f^j(x)) )
$
where $k=k(x,t,s)\ge 0$ is determined by
$
\sum_{j=0}^{k-1} \tau(f^j(x)) \leq t+s < \sum_{j=0}^{k} \tau(f^j(x)).
$
Observe that if $X_s(x,t)\in A$ for $s>e_A(x,t)$ and $k=k(x,t,s)$ is defined as above 
then necessarily $f^k(x)\in I$. Reciprocally, if $(x,t)\in\Sigma_\tau$ and $f^k(x)\in I$ 
then there exists $s<\sum_{j=0}^{k} \tau(f^j(x)) - t$ such that $X_s(x,t)\in A$.
Now, since we are dealing with a cylinder $A= I \times [t_1,t_2]$ then $e_A(x,t)= t_2-t$ 
and so
\begin{align}
n_A(x,t)  
		& = \big[\sum_{j=0}^{n_I(x)-1} \tau(f^j(x)) - t \big]+t_1 \nonumber \\
		& = \sum_{j=0}^{n_I(x)-1} \tau(f^j(x)) + e_A(x,t) - m([t_1,t_2]) \label{eq:exit}
\end{align}
for any $(x,t)\in A$.
Given $n\ge 1$ consider the sets $I_n=\{x\in I : n_I(x)=n\}$ and it is clear that 
$I=\cup I_n$ modulo a zero measure set with respect to $\mu$. For $0\le k \le n-1$ set also 
$I_{n,k} = f^{k}(I_n)$. 
Since all points in $I_{n,k}$ have $n-k$ as first hitting time to $A$ and $f$ is invertible then 
$f^{-k}(I_{n,k})=I_n$ and, consequently, by $f$-invariance of the measure it holds $\mu(I_{n,k})=\mu(I_n)$. 
Since $\mu$-almost every $x$ will  eventually visit the set $I\subset \Sigma$, $\mu$ is ergodic 
and $I_{n,k}$ is family of disjoint sets then $\mu(\cup_n \cup_{0\le k \le n-1} I_{n,k})=1$.
Hence, 
\begin{align}
\int_\Sigma \tau \; d\mu
	& = \sum_{n\ge 1} \sum_{k=0}^{n-1} \int_{I_{n,k}} \, \tau \; d\mu
	 = \sum_{n\ge 1} \sum_{k=0}^{n-1} \int_{I_{n}} \, \tau\circ f^k \; d\mu \nonumber \\
	& = \sum_{n\ge 1}  \int_{I_{n}} \, \sum_{k=0}^{n-1} \tau\circ f^k \; d\mu
	= \int_I \sum_{k=0}^{n_I(x)-1} \tau\circ f^k \; d\mu. \label{eq:integ2}
\end{align}
In consequence using equations \eqref{eq:exit} and \eqref{eq:integ2} together it follows that
\begin{align}
\int_A   n_A(x,t) \; d\bar\mu
	& =  \frac{1}{\int_\Sigma \tau\,d\mu} \; \int_{t_1}^{t_2} \int_I  \Big[  \sum_{j=0}^{n_I(x)-1} \tau \circ f^j 
		 \Big] 
				\; d\mu \, dt   \nonumber \\
	& + \int_A e_A(x,t)   \, d\bar \mu  - m([t_1,t_2]) \, \bar\mu(A) \nonumber 
\end{align}
or, in other words,
\begin{equation}\label{eq:Stat1}
\int_A   n_A(x,t) \; d\bar\mu
	 =   \int_A e_A(x,t)   \, d\bar \mu + m([t_1,t_2]) \, \big( 1- \bar\mu(A) \big).
\end{equation}
If $\bar\mu_A$ denotes as before the normalized probability measure $\frac{\bar\mu\mid_A(\cdot)}{\bar\mu(A)}$ the later becomes
\begin{equation}\label{eq:Stat2}
\int_A   n_A(x,t) \; d\bar\mu_A
	 =   \int_A e_A(x,t)   \, d\bar \mu_A + \big( 1- \bar\mu(A) \big) \frac{1}{\mu(I)} \, \int_\Sigma \tau \,d\mu
\end{equation}
which proves the first equality in the theorem. A simple integral computation shows that
$\int_{t_1}^{t_2} (t_2 -t) \,dt = m([t_1,t_2])^2 / 2$ leading to
$
\int_A e_A(x,t)  \,d\bar\mu
	= \bar\mu(A) \, \frac{m([t_1,t_2])}{2}.
$ 
So $\int_A e_A(x,t)  \,d\bar\mu_A=\frac{m([t_1,t_2])}{2}$ and replacing
$\bar\mu(A)= \frac{\mu(I)}{\int \tau\, d\mu} \, m([t_1,t_2])$ in equation~\eqref{eq:Stat2} we conclude 
\begin{align*}
\int_A   n_A(x,t) \; d\bar\mu_A 
	& =   \frac{m([t_1,t_2])}{2} + \frac{1}{\mu(I)} \int_\Sigma \bigg( \tau  - (t_2-t_1) \chi_I \bigg) \, d\mu
\end{align*}
where $\chi_I$ denotes the indicator function of the set $I \subset \Sigma$. This proves the second 
equality in the theorem. In particular, it follows from the previous expression that if the set $A$ is kept unchanged the mean return time varies linearly with $\int_\Sigma \tau \;d\mu$. This finishes the proof of the 
theorem.
\end{proof}

Some comments are in order. The first one concerns the ergodicity assumption.
It is a simple consequence of Birkhoff's ergodic theorem and the integrability of the roof function
that $\bar \mu$ is ergodic for the flow if and only if $\mu$ is ergodic for $f$. If ergodicity fails it follows from
the ergodic decomposition theorem that for $\mu$ almost every $x\in \Sigma$ there are ergodic
probability measures $\mu_{x}$ on $\Sigma$ so that 
$\mu = \int \mu_{x} \; d\mu$. In such case $\bar \mu = \int \bar \mu_{x} d\bar \mu$ is the ergodic 
decompositon for $\bar \mu$, where $\bar\mu_x$ are the almost everywhere 
defined ergodic measures $\bar \mu_x = (\mu_{x} \times m) / \int \tau \, d\mu_{x}$. Using this we recover analogous
expression for the mean return time as in Theorem~\ref{thm:Kac} without the ergodicity assumption.

A second remark is related with the heigh of the cylinders $A$. Indeed, the role of the constants
$t_1$ and $t_2$ was not important to guarantee recurrence to the set $A$ since in our setting
$\Sigma$ is a global cross-section and recurrence is obtained simply by assuming the projection 
$I$ has positive $\mu$-measure. Hence, as a consequence of our strategy we deduce the following 
recurrence estimates for subsets of the cross-section.

\begin{maincorollary}
Let $f: \Sigma\to \Sigma$ be a measurable invertible map on a topological space $\Sigma$, $\mu$
be an $f$-invariant ergodic probability measure and let $\tau \in L^1(\mu)$ be a roof function, and let 
$(X_t)_t$ be the suspension semiflow over preserving $\bar \mu$.
Then, for any positive $\mu$-measure set $A\subset \Sigma$,
\begin{align*}
\int_A   n_A(x,0) \; d\mu_A 
	& = \frac{1}{\mu(A)} \int_\Sigma  \tau \, d\mu
\end{align*}
where $\mu_A$ of the normalization of the measure $\mu\mid_A$ in the cross-section $\Sigma$.
In particular,  if $f_A : A \to A$ denotes the first return time map then 
$$
\int_A   n_A(x,0) \; d\mu_A 
	= \frac{h_{\mu_A}(f_A)}{h_{\bar\mu}(X_1)}
$$
is the entropy of the first return time map quotiented by the entropy of the flow. 
\end{maincorollary}

\begin{proof}
For simplicity reasons we shall denote also by $A$ the set $A\times\{0\} \subset \Sigma_\tau$.
it is clear that $e_A(\cdot, \cdot) \equiv 0$. Hence
It follows e.g. from equation~\eqref{eq:exit} and ~\eqref{eq:integ2} (taking $I=A$) that 
$
n_A(x,0)  = \sum_{j=0}^{n_A(x)-1} \tau(f^j(x))
$
and we deduce that
$
\int_A n_A(x,0) \, d\mu = \int_\Sigma \tau \, d\mu,
$
from which expression the result immediately follows. 

In addition, on the one hand  Abramov~\cite{Ab1} proved that the entropy $h_{\mu_A}(f_A)$ 
of the first return time map $f_A$ with respect to $\mu_A$ satisfies 
$
h_\mu(f)=h_{\mu_A}(f_A) \, \mu(A).
$
On the other hand, the formula established for the time-$t$ map of the flow with relation
with the base map obtained by Abramov~\cite{Ab2} 
$$
h_{\bar\mu} (X_t) = \frac{|t| \, h_\mu(f)}{\int_\Sigma \tau \, d\mu}  
$$
and so
$$
\int_A   n_A(x,0) \; d\mu_A  
	= \frac1{\mu(A)} \int_\Sigma \tau \, d\mu  
	= \frac{h_{\mu_A}(f_A)}{h_\mu(f)} \, \frac{h_\mu(f)}{h_{\bar\mu}(X_1)}
	= \frac{h_{\mu_A}(f_A)}{h_{\bar\mu}(X_1)}.
$$
This finishes the proof of the corollary.
\end{proof}

\begin{remark}
Let us mention that the assumption of Theorem~\ref{Hel} 
is satisfied in the case of suspension flows and cylinder sets.
In fact, if $A=I \times [t_1,t_2] \subset \Sigma_\tau$ and $s>0$ is small then 
$A_s=I \times [t_2-s, t_2]$ and $\tilde A_s=I \times \{t_1\}$. 
In particular $\bar\mu(A_s) \to 0$ as $s$ tends to zero. Since $A_s\subset A$ then the return time $r_A$
defined in \cite{Hel} coincides with the definition given in equation~\ref{def:entrance} and it follows that
$$
\lim_{s \to 0} \frac{1}{s} \int_{A_s} r_A (z) \, d\bar\mu
	= 1 - \bar\mu (A).
$$
\end{remark}

Now, let us observe that the case of flows cannot be obtained directly
from the discrete time setting since it reflects the escaping times and the width of the cylinders 
as shown by the following immediate consequence of the theorem, in which we also use the mean 
return times to compute entropy of a measure in the cross section.

\begin{corollary}
Let $((X_t)_t,\bar\mu)$ be the suspension semiflow over $(f,\mu,\tau)$ as in Theorem~\ref{thm:Kac} and
assume that the roof function $\tau$ is constant.
If $A= I \times [0,\tau] \subset \Sigma\times \mathbb R$ is a full cylinder set in $\Sigma_\tau$ 
with positive $\bar\mu$-measure then 
$$
\int_A n_A(x,t) \; d\bar\mu_A 
	 =  \frac{1}{\mu(I)} \int_\Sigma \bigg( \tau  - \frac\tau{2} \chi_I \bigg) \, d\mu
	= \frac{\tau}{\mu(I)} \left( 1   - \frac{\mu(I)}{2}  \right). 
$$
\end{corollary}

Despite the fact that cylinder sets arise naturally for flows (e.g. from the tubular neighborhood theorem)
sometimes we are interested in the return times to geometric objects as balls which can be used to
study dimension of measures among other relevant dynamical quantities. 
For that reason, we will now extend our result for a more general class of sets that include balls.
Let $\cA$ be the family of closed sets whose boundary are graphs of functions over $\Sigma$. More
precisely, $A\in \cA$ if and only if there are measurable functions $h_1, h_2: \pi_1(A) \to \mathbb R_0^+$ such that $A=\{(x,t)\in \Sigma_\tau : h_1(x) \le t \le h_2(x) \}$, where $\pi_1: \Sigma_\tau \to\Sigma$ 
denotes the natural projection on $\Sigma$. For simplicity we shall assume $h_1(x)\le h_2(x) \le \tau(x)$ for 
every $x\in \Sigma$ ands denote such sets by $A_{h_1,h_2}$.
Just as a remark, for the purpose of recurrence properties to these sets the return times for the flow 
coincide with the ones if one had considered sets of the form 
$A=\{(x,t)\in \Sigma_\tau : h_1(x) < t < h_2(x) \}$.
We can now state our next result.

\begin{maintheorem}\label{thm:Kac2}
Let $(X_t)_t$ be a suspension semiflow associated to $(f,\Sigma,\tau)$ and let $\bar \mu$ be the ergodic
$(X_t)_t$-invariant probability measure associated to the $f$-invariant ergodic probability measure 
$\mu$. Given $A=A_{h_1,h_2}\in \mathcal A$ with $\bar \mu$-positive measure then
\begin{align*}
\int_A n_A(x,t) \; d\bar\mu_A
	& =  \int_{A} e_A (x,t) \; d\bar\mu_A 
	+\frac{ \int_I   h(x) \, \tau^{n_I}(x) \; d\mu }{ \int_I   h(x)  \; d\mu } \\
	 & + \frac{\int_I h(x)  [h_1(f^{n_I(x)}(x)) - h_2(x)] \; d\mu }{\int_I   h(x)  \; d\mu }
\end{align*}
where $h(x)=h_2(x)-h_1(x)$,  $ \tau^{n_I}(x):=\sum_{k=0}^{n_I(x)-1} \tau\circ f^k(x)$ and 
also 
$$
\int_{A} e_A (x,t) \; d\bar\mu_A 
	= \frac{ \int_I   h(x) \, \big [ \frac{h_2(x)-h_1(x)}{2} \big] \; d\mu }{ \int_I   h(x)  \; d\mu }.
$$ 
\end{maintheorem}

\begin{proof}
Set $A=A_{h_1,h_2}$ and $I=\pi_1(A)$ for simplicity. Given $(x,t)\in A$ then clearly the escape time is
given by $e_A(x,t)=h_2(x)-t$. Moreover, if there exists  $s\ge 0$ so that
$
X_s(x,t)
	=\big(f^k(x), t+s-\sum_{j=0}^{k-1} \tau(f^j(x)) \big)
	\in A 
$
then $f^k(x)\in I$ and $k=k(x,t,s)\ge 0$ given as after \eqref{eq:1}.
Therefore, similarly to the proof of Theorem~\ref{thm:Kac}, for $\bar\mu$-almost every $(x,t)\in A$ we get
\begin{align*}
n_A(x,t)  
		& = \sum_{j=0}^{n_I(x)-1} \tau(f^j(x)) + e_A(x,t) - h_2(x) +h_1(f^{n_I(x)}(x)), \label{eq:exit2}
\end{align*}
we write $I=\cup I_n$ $(\text{mod} \; \mu)$ and
notice, by ergodicity, the family  $I_{n,k}=f^k(I_n)$ pairwise of disjoint sets verifies 
$\mu(\cup_n \cup_{0\le k \le n-1} I_{n,k})=1$. Hence, a simple computation shows that 
$$
\bar\mu(A)=\frac{1}{\int \tau \,d\mu} \int_I h(x)\, d\mu
$$ 
and also
\begin{align*}
\int_A   n_A(x,t) \; d\bar\mu
	& = \frac{1}{\int \tau \,d\mu} \; \int_{I}   h(x) \, \tau^{n_I}(x) 
				\; d\mu 
	 + \frac{1}{\int \tau \,d\mu} \; \int_{I} \int_{h_1(x)}^{h_2(x)} [  h_2(x)-t ]\; dt\, d\mu \\
	  &  + \frac{1}{\int \tau \,d\mu} 
	  \; \int_{I} \int_{h_1(x)}^{h_2(x)}  [h_1(f^{n_I(x)}(x)) - h_2(x)] \; \; dt\, d\mu 
\end{align*}
and the result follows by simple computations.
\end{proof}

\section{Final remarks}

\subsubsection*{'Almost' cylinders}

The first remark is that computations are above clearly simpler in the case where $A_{h_1,h_2}$ has 'parallel sides' meaning $h_2(x)=h_1(x)+c$ for some $c>0$.  In this case, using again the notation $I=\pi_1(A)$ and noticing $h(x)=c$ for every $x\in I$ it follows that $\bar\mu(A)= c \, \mu(I) / \int \tau \,d\mu$, that 
$
\int_A   e_A(x,t) \; d\bar \mu_A
		= \frac{c}{2}  
$
and, using that the first return time map $f^{n_I}: I \to I$ preserves the (ergodic) normalized probability 
measure $\mu_I$
\begin{align*}
\int_A n_A(x,t) \; d\bar\mu_A
	& =  \frac{c}2
	+ \frac{1}{ \mu(I) } \int_{\Sigma}   \tau \; d\mu
	+ \frac{ 1 }{ \mu(I) } 
	 \int_I [h_1(f^{n_I(x)}(x)) - h_1(x) -c ] \; d\mu \\
	& =  \frac{c}2
	+ \frac{1}{ \mu(I) } \int_{\Sigma}   (\tau -c \chi_I) \; d\mu
	+ \int_I [h_1(f^{n_I(x)}(x)) - h_1(x)] \; d\mu_I \\
	& =  \frac{c}2
	+ \frac{1}{ \mu(I) } \int_{\Sigma}   (\tau -c \chi_I) \; d\mu
\end{align*}

\vspace{.2cm}
\subsubsection*{Suspension semiflows}

We should mention that our results hold for suspension semiflows $(X_t)_{t\ge 0}$ associated to 
the suspension for non-invertible maps $f :\Sigma \to \Sigma$. Given an ergodic measure $\mu$, for $I\subset \Sigma$ we 
can decompose
$$
\Sigma = \bigcup_{n\ge 1} [ I_n \cup I_n^* ] \quad  (\mu_\Sigma \!\!\!\!\! \mod 0)
$$
where $I_n=I \cap \{ n_I(\cdot) = n\}$ and $I_n^*= (\Sigma \setminus I)  \cap \{ n_I(\cdot) = n\}$, and all
elements in the previous union are pairwise disjoint. Moreover, for every $n\ge 1$ we get 
$\mu (I_n^*) = \mu (f^{-1}(I_n^*)) = \mu(I_{n+1})+\mu (I_{n+1}^*)$
and consequently $\mu (I_n^*)=\sum_{k\ge n+1} \mu (I_k)$
and so $1=\mu (\Sigma)= \sum_{k\ge 1} k \mu (I_k)=\int n_I(\cdot) \, d\mu$.
Taking $I_{n,j}=f^j(I_n)$ for $0\le j \le n-1$, since $I_{n,j} \subset I_{n-j}^*$,
it follows that
\begin{align*}
\mu  \big( \bigcup_{n \ge 1} \bigcup_{j=0}^{n-1}  I_{n,j} \big)
	& = \mu  \big( \bigcup_{n \ge 1} \bigcup_{j \ge 0}  I_{n+j,j} \big) \\
	& = \sum_{n\ge 1} 
		\big[ \mu  \big(  I_{n,0} \big) 
		+  \mu  \big(  \bigcup_{j \ge 0}  I_{n+j,j} \big) \big] \\
	& = \sum_{n\ge 1} \big[ \mu (  I_{n} ) + \mu (  I^*_{n} ) \big] =1.
\end{align*}
Thus, as in equation~\eqref{eq:integ2} and obtain
$
\int_\Sigma \tau \; d\mu
	= \int_I \sum_{k=0}^{n_I(x)-1} \tau\circ f^k \; d\mu.
$
and the same computations as in the proof of Theorems~\ref{thm:Kac} and ~\ref{thm:Kac2} 
follow straightforward.

\vspace{.3cm}
\subsection*{Acknowledgements} This work was supported by a CNPq-Brazil postdoctoral fellowship at
University of Porto. The author is grateful to 
M. Bessa, F. Rodrigues and J. Rousseau for some comments on the manuscript. The 
author is deeply grateful to J.-R. Chazottes for his comments and for providing the reference~\cite{Hel}. 
The author is also grateful to the organizers of the event ``Probability in Dynamics at UFRJ" where part 
of this work was developed. 


\end{document}